\documentclass[preprint,12pt]{elsarticle}

\usepackage{graphicx,amsmath,amssymb,mathrsfs,amsfonts,amsthm,color,slashbox,amsbsy,ragged2e}
\usepackage{hyperref}

 \usepackage{lineno}

\usepackage[figuresright]{rotating}

\newcommand{\bq}{\begin{equation}}
\newcommand{\eq}{\end{equation}}
\newcommand{\bqs}{\begin{equation*}}
\newcommand{\eqs}{\end{equation*}}
\newcommand{\bqa}{\begin{eqnarray}}
\newcommand{\eqa}{\end{eqnarray}}
\newcommand{\bqas}{\begin{eqnarray*}}
\newcommand{\eqas}{\end{eqnarray*}}
\newcommand{\bc}{\begin{cases}}
\newcommand{\ec}{\end{cases}}
\newcommand{\bt}{\begin{thm}}
\newcommand{\et}{\end{thm}}
\newtheorem{theorem}{Theorem}[section]
\newtheorem{lemma}[theorem]{Lemma}

\newtheorem{corollary}[theorem]{Corollary}

\newtheorem{definition}[theorem]{Definition}

\begin{document}
\begin{frontmatter}


\title{Stochastic comparisons of series and parallel systems with heterogeneous  components}
\author[a]{Esmaeil Bashkar}
\author[a]{\corref{cor1}Hamzeh Torabi}
\ead{htorabi@yazd.ac.ir}
\author[b]{Majid Asadi}
\cortext[cor1]{Corresponding author at: Department of Statistics, Yazd University, Yazd, Iran.}
\address[a]{Department of Statistics, Yazd University, Yazd, Iran}
\address[b]{Department of Statistics, University of Isfahan, Isfahan, Iran}


%

\begin{abstract}
In this paper, we discuss stochastic comparisons of parallel  systems with independent  heterogeneous exponentiated Nadarajah-Haghighi (ENH)
components in terms of the usual stochastic order,
dispersive order, convex transform order and the likelihood ratio order. In the presence of the Archimedean copula, we study stochastic comparison of series dependent systems in terms of the usual stochastic order. Due to the great flexibility of the failure rate function of the ENH, it thus provides a good alternative to many existing life distributions in modeling
positive real data sets in practice. In particular, it can be an interesting alternative to the well-known three-parameter
exponentiated Weibull (EW) distribution.
\end{abstract}

\begin{keyword}
Exponentiated Nadarajah-Haghighi distribution  \sep Stochastic ordering  \sep Majorization \sep Parallel system \sep Series system.



\end{keyword}
\end{frontmatter}


\section{Introduction}

Let $ X_{1:n}\leq \ldots \leq X_{n:n} $ denote the order statistics arising from random variables $ X_{1}, \ldots , X_{n} $. Order statistics play a prominent rule  in the reliability theory,  life testing, operations research and other related areas. In reliability theory, the $ k $th order statistic coresponds to the lifetime of a $ (n-k+1) $-out-of-$ n $ system. In particular, $  X_{1:n} $ and $X_{n:n} $ correspond to the lifetimes of series and parallel systems, respectively. Various researchers have studied
the stochastic comparisons for the lifetimes of the series and parallel systems. 
For example \cite{k}, \cite{fz13}, \cite{to15}, \cite{ll}, \cite{tk15} and \cite{fb15} deal with the case of heterogeneous Weibull distributions, \cite{fz15} and \cite{kc16} deal with the case of heterogeneous exponentiated Weibull (EW) distributions, \cite{bee} deals with the case of heterogeneous  generalized exponential (GE) distributions and \cite{gup} deals with the  case of heterogeneous Fr$\grave{\rm e}$chet distributions. A recent review on the topic can be also found in \cite{bz}. 

A new generalization of the exponential distribution as an alternative to the gamma, Weibull and GE distributions was proposed by Nadarajah and Haghighi \cite{nh}. Its cumulative function is given by
\begin{equation}\label{nh1}
F(x) = 1-\exp\{ 1-(1+\lambda x)^{\alpha} \},\quad x>0,
\end{equation}
where $ \lambda>0 $ is the scale parameter, and $ \alpha>0 $ is the shape parameter. Lemonte \cite{lem} proposed a new three-parameter generalization of the exponential distribution on the basis of the
distribution proposed by Nadarajah and Haghighi \cite{nh}. The new family of distributions is rather simple and is constructed by raising the cumulative function given in Eq. \eqref{nh1} to
an arbitrary power, $ \beta>0 $ say. The new cumulative function is given by
\begin{equation}\label{l1}
G(x) = [1-\exp\{ 1-(1+\lambda x)^{\alpha} \}]^{\beta},\quad x>0,
\end{equation}

where the parameters $ \alpha>0 $ and $ \beta>0 $ control the shapes of the distribution, and the parameter $ \lambda>0 $ is the scale
parameter. We shall refer to the new distribution given in \eqref{l1} as the exponentiated NH (ENH) distribution. If a random
variable $ X $ has the ENH distribution, then we write $ X \thicksim {\rm ENH}(\alpha, \lambda, \beta ) $. Clearly, if $ \beta = 1 $, the ENH distribution reduces to
the NH distribution. For $ \alpha = 1 $, we obtain the GE distribution proposed by Gupta and Kundu \cite{gk}. We have the exponential distribution when $ \alpha = \beta =1 $.

Similarly to the exponentiated Weibull model, the ENH failure rate function can have the following four forms depending on its shape
parameters: (i) increasing; (ii) decreasing; (iii) unimodal (upside-down bathtub); (iv) bathtub-shaped. Therefore, it can be used quite effectively in analyzing lifetime data. Additionally, the new ENH model can be used as an alternative to the
EW distribution constructed by Mudholkar and Srivastava \cite{ms}. In Section 3, we discuss stochastic comparisons of parallel systems with independent  heterogeneous ENH components in terms of the usual stochastic order,
dispersive order, convex transform order and the likelihood ratio order. In the presence of the Archimedean copula, we study stochastic comparison of series dependent systems in terms of the usual stochastic order. To continue our discussion, we need definitions
of some stochastic orders and the concept of majorization which is given in Section 2 of the paper.

\section{The basic definitions and some prerequisites}\label{2}
In this section, we recall some notions of stochastic orders, majorization and related orders and some useful lemmas, which are helpful for proving our main results. Throughout this paper, we use the notations $ \Bbb R = (-\infty,+\infty) $ and $ \Bbb R_{++} = (0,+\infty) $ and
 the term increasing means nondecreasing and decreasing
means nonincreasing.

Let $  X $ and $ Y $ be two non-negative random variables with distribution functions $ F $ and $ G $, density functions $ f $ and $ g $, the survival functions $ \bar{F}=1 - F  $ and $ \bar{G}=1 - G  $, the right continuous inverses\footnote{The right continuous inverse of an increasing function $ h $ is defined as $ h^{-1}(u)=\sup\{x\in \Bbb R : h(x)\leq u\} $} $ F^{-1} $ and $ G^{-1} $ of $ F $ and $ G $, and hazard rate functions $ h_{F}=\frac{f}{\bar{F}} $ and $ h_{G}=\frac{g}{\bar{G}} $,  respectively.


The following definition introduces some well-known orders that compare skewness
of probability distributions.
\begin{definition}
{\rm $ X $ is said to be smaller than $ Y $ in the 
\begin{itemize} 
\item[{\rm(i) }] convex transform order denoted by $ X\leq_{\rm c}Y $ if $ G^{-1}F(x) $ is convex in $ x \geq0 $;
\item[{\rm(ii)}]  Lorenz order denoted by $ X\leq_{\rm Lorenz}Y $ if
\begin{equation}
\frac{1}{E(X)}\int_0^{F^{-1}(u)}x dx \geq \frac{1}{E(Y)}\int_0^{G^{-1}(u)}x dx, \quad \forall u \in (0,1].
\end{equation}
\end{itemize}}
\end{definition}
The convex transform order implies the Lorenz order which,
in turn, implies the order between the corresponding the coefficients of
variations.

The following definition gives some well-known orders that compare the
dispersion of two random variables.  
\begin{definition}
{\rm $ X $ is said to be smaller than $ Y $ in the 
\begin{itemize}

\item[{\rm(i) }] dispersive order, denoted by $ X\leq_{\rm disp}Y $, if $  F^{-1}(\beta)-F^{-1}(\alpha)\leq G^{-1}(\beta)-G^{-1}(\alpha)$ for all $ 0<\alpha\leq\beta<1 $,
\item[{\rm(ii)}] right-spread order, denoted by $ X\leq_{\rm RS}Y $, if 
\begin{equation}
\int_{F^{-1}(u)}^{\infty} \bar{F}(x) dx \leq \int_{G^{-1}(u)}^{\infty} \bar{G}(x) dx, \quad \forall u \in (0,1).
\end{equation}
\end{itemize}}
\end{definition}
It is well-known that the dispersive order implies the right-spread order which,
in turn, implies the order between the corresponding variances.

The next definition introduces some well-known orders that compare the magnitude
of two random variables.
\begin{definition}
{\rm $ X $ is said to be smaller than $ Y $ in the 
\begin{itemize}
\item[\rm (i)] stochastic order, denoted by $ X \leq_{\rm st} Y $, if $ \bar{F}(x)\leq\bar{G}(x) $ for all  $ x $;
\item[{\rm(ii)}] likelihood ratio order, denoted by $ X\leq_{\rm lr}Y $, if $ g(x)/f (x) $ is increasing in $ x\in \Bbb R_{++} $;
\item[{\rm(iii)}] hazard rate order, denoted by $ X\leq_{\rm hr}Y $, if $ h_{F} (x) \geq h_{G}(x) $ for all $ x $.
\end{itemize}}
\end{definition}
Note that the likelihood ratio order implies the hazard rate order, and the hazard
rate order implies the usual stochastic order. Moreover, for non-negative random
variables, the dispersive order implies the usual stochastic order. For a comprehensive
discussion on various stochastic orderings, one may refer to \cite{ss} and \cite{lll}.

A real function $ \phi $ is $n$-monotone on $ (a, b) \subseteq (-\infty,+\infty) $ if $ (-1)^{n-2}\phi^{(n-2)} $ is
decreasing and convex in $ (a, b) $ and $ (-1)^{k}\phi^{(k)}(x)\geq0 $ for all $ x \in (a, b), k =
0, 1, \ldots , n - 2 $, in which $\phi^{(i)}(.)$ is the $i$th derivative of $\phi(.)$.
For a $n$-monotone $ (n \geq 2) $ function $ \phi : [0,+\infty) \longrightarrow [0, 1] $ with $ \phi(0) = 1 $ and $ \lim_{x\rightarrow+\infty} \phi(x) = 0 $, let $ \psi = \phi ^{-1} $ be the
pseudo-inverse, then
\begin{equation*}
C_{\phi}(u_1, \ldots , u_n) = \phi (\psi(u_1) + \ldots + \psi(u_n)), \quad \text{for all}\, u_{i} \in [0, 1], i = 1, \ldots , n,
\end{equation*}
is called an Archimedean copula with the generator $ \phi $. Archimedean copulas
cover a wide range of dependence structures including the independence copula with the
generator $ \phi(t)=e^{-t} $. For more on Archimedean copulas, readers may refer to \cite{nels}
and \cite{ncn}.

The concepts of majorization of vectors and Schur convexity of functions
will also be needed.
For some extensive and comprehensive discussions on the
theory of these orders and their applications, one can see \cite{met}. Let us recall that the notation $ x_{(1)}\leq x_{(2)}\leq ...\leq x_{(n)}$  is used to denote the increasing arrangement of the components of the vector $ \boldsymbol{x} = (x_{1}, \ldots , x_{n})$.

\begin{definition}
{\rm The vector $ \boldsymbol{x} $ is said to be
\begin{itemize}
\item[ (i)] weakly submajorized by the vector $ \boldsymbol{y} $ (denoted by $ \boldsymbol{x}\preceq_{\rm w}\boldsymbol{y} $) if
$\sum_{i=j}^{n}x_{(i)}\leq \sum_{i=j}^{n}y_{(i)}$ for all $j = 1, \ldots , n $,

\item[ (ii)] weakly supermajorized by the vector $ \boldsymbol{y} $ (denoted by $ \boldsymbol{x}\mathop \preceq \limits^{{\mathop{\rm w}} }\boldsymbol{y} $) if $ \sum_{i=1}^{j}x_{(i)}\geq \sum_{i=1}^{j}y_{(i)} $ for all $ j = 1, \ldots , n $,

\item[ (iii)] majorized by the vector $ \boldsymbol{y} $ (denoted by $ \boldsymbol{x}\mathop \preceq \limits^{{\mathop{\rm m}} }\boldsymbol{y} $) if $ \sum_{i=1}^{n}x_{i}= \sum_{i=1}^{n}y_{i}$ and $\sum_{i=1}^{j}x_{(i)}\geq \sum_{i=1}^{j}y_{(i)}$ for all $j = 1, \ldots , n-1 $.
\end{itemize}}
\end{definition}
\begin{definition}
{\rm A real valued function $ \varphi $ defined on a set $ \mathscr{A}\subseteq {\Bbb R}^{n} $ is said to be Schur-convex (Schur-concave) on $ \mathscr{A} $ if

\[
\boldsymbol{x} \mathop \preceq \limits^{{\mathop{\rm m}} }\boldsymbol{y} \quad \text{on}\quad \mathscr{A} \Longrightarrow \varphi(\boldsymbol{x})\leq (\geq)\varphi(\boldsymbol{y}).
\]}
\end{definition}

\begin{lemma}[\cite{met}, Theorem 3.A.8]\label{mkl}
{\rm For a function $ l $ on $ \mathscr{A}\in \Bbb R^{n} $, $ \boldsymbol{x}\preceq_{\rm w}(\mathop \preceq \limits^{{\mathop{\rm w}} })\boldsymbol{y} $ implies $ l(\boldsymbol{x}) \leq l(\boldsymbol{y}) $ if and only if
it is increasing (decreasing) and Schur-convex on $ \mathscr{A} $.}
\end{lemma}

\section{Main results}
Firstly, we introduce the following lemma, which will be needed to prove our main results.
\begin{lemma}\label{mlee}
{\rm Let the function $ g : (1, \infty)\longrightarrow (0, \infty) $ be defined as
\[
g(x)=\frac{x e^{1-x}}{1-e^{1-x}}.
\]
Then $  g(x) $ is a decreasing function on $ (1, \infty) $.}
\end{lemma}
\begin{proof}
It is easy to see that
\[
g^{\prime}(x)=\dfrac{f(x)}{(1-e^{1-x})^{2}}
\]
where $ f(x)=e^{1-x}-e^{2(1-x)}-xe^{1-x} =e^{1-x}(1-x-e^{1-x})<0 $ for $ x>1 $.  Hence $ g(x) $ is a decreasing function on $ (1, \infty) $. 
\end{proof}
\subsection{Mutually independent samples}
In this section, we provide some comparison results on the lifetimes of parallel systems arising from independent heterogeneous
ENH random variables. The following result considers the comparison on the lifetimes of parallel systems in terms
of the usual stochastic order with respect to the shape parameter $ \alpha $.
\begin{theorem}\label{e1}
{\rm Let $ X_{1}, \ldots , X_{n} $ ($ X^{*}_{1}, \ldots , X^{*}_{n} $) be independent random variables with $ X_{i} \thicksim {\rm ENH}(\alpha_{i}, \lambda, \beta )$ ($ X^{*}_{i} \thicksim {\rm ENH}(\alpha^{*}_{i}, \lambda, \beta )$), $i = 1, \ldots, n $. Then, for any $ \lambda, \beta>0 $, we
have
\[
 (\alpha_{1}, \ldots , \alpha_{n})\mathop \succeq \limits^{{\mathop{\rm w}} }(\alpha^{*}_{1}, \ldots , \alpha^{*}_{n})  \Longrightarrow 
 X_{n:n}\geq_{\rm st}X^{*}_{n:n}.
\]
}
\end{theorem}
\begin{proof}
{\rm The distribution function of $ X_{n:n} $ can be written as
\begin{equation*}
G_{X_{n:n}}(x)= \prod_{i=1}^{n} [1-e^{1-(1+\lambda x)^{\alpha_i}}]^{\beta}.
\end{equation*}
Using Lemma \ref{mkl}, it is enough to show that the function $ G_{X_{n:n}}(x) $ is Schur-concave and increasing in $ \alpha_{i} $'s. The partial derivatives of $ G_{X_{n:n}}(x) $ with respect to $ \alpha_{i} $ is given by
\[
\frac{\partial G_{X_{n:n}}(x)}{\partial \alpha_{i}}=\frac{\beta \log(1+\lambda x)(1+\lambda x)^{\alpha_i}e^{1-(1+\lambda x)^{\alpha_i}}}{1-e^{1-(1+\lambda x)^{\alpha_i}}}G_{X_{n:n}}(x)\geq 0,
\]
So, we have that $ G_{X_{n:n}}(x) $ is increasing in each $ \alpha_{i} $.

To prove its Schur-concavity, it follows from Theorem 3.A.4. in \cite{met} that we
have to show that for $ i\neq j $,
\[
(\alpha_{i}-\alpha_{j})\bigg(\frac{\partial G_{X_{n:n}}(x)}{\partial \alpha_{i}}-\frac{\partial G_{X_{n:n}}(x)}{\partial \alpha_{j}} \bigg)\leq 0,
\]
that is, for $ i\neq j $,
\[
(\alpha_{i}-\alpha_{j}) G_{X_{n:n}}(x)\beta \log(1+\lambda x)\times
\]
\begin{equation}\label{sq32}
\bigg(\frac{(1+\lambda x)^{\alpha_i}e^{1-(1+\lambda x)^{\alpha_i}}}{1-e^{1-(1+\lambda x)^{\alpha_i}}}-\frac{(1+\lambda x)^{\alpha_j}e^{1-(1+\lambda x)^{\alpha_j}}}{1-e^{1-(1+\lambda x)^{\alpha_j}}}\bigg)\leq 0.
\end{equation}
It is obvious that $ (1+\lambda x)^{\alpha} $ is increasing in $ \alpha $. So, from Lemma \ref{mlee}, we conclude that the composite function $ \frac{(1+\lambda x)^{\alpha}e^{1-(1+\lambda x)^{\alpha}}}{1-e^{1-(1+\lambda x)^{\alpha}}} $ is decreasing in $ \alpha $, from which it follows
that \eqref{sq32} holds. This completes the proof of the required result.}
\end{proof}
We now generalize Theorem \ref{e1} to a wider range of the shape parameters
as follows.
\begin{theorem}\label{e11}
{\rm Let $ X_{1}, \ldots , X_{n} $ ($ X^{*}_{1}, \ldots , X^{*}_{n} $) be independent random variables with $ X_{i} \thicksim {\rm ENH}(\alpha_{i}, \lambda, \beta )$ ($ X^{*}_{i} \thicksim {\rm ENH}(\alpha^{*}_{i}, \lambda, \beta )$), $i = 1, \ldots, n $. Then, for any $ \lambda, \beta>0 $, if $ (\alpha_{1}, \ldots , \alpha_{n})\leq(\alpha^{*}_{1}, \ldots , \alpha^{*}_{n}) $, that is, $ \alpha_{i}\leq \alpha^{*}_{i} $, $ i = 1, \ldots, n  $, we have $  X_{n:n}\geq_{\rm st}X^{*}_{n:n} $.
}
\end{theorem}
\begin{proof}
Using the definition of the usual stochastic order and the
fact that $ G_{X_{n:n}}(x) $ is increasing in each $ \alpha_{i} $, the required
results follow readily.
\end{proof}
 The following result considers the comparison on the lifetimes of parallel systems in terms
of the usual stochastic order when two sets of scale parameters weakly majorize each other.
\begin{theorem}\label{e2}
{\rm Let $ X_{1}, \ldots , X_{n} $ ($ X^{*}_{1}, \ldots , X^{*}_{n} $) be independent random variables with $ X_{i} \thicksim {\rm ENH}(\alpha, \lambda_{i}, \beta )$ ($ X^{*}_{i} \thicksim {\rm ENH}(\alpha, \lambda^{*}_{i}, \beta )$), $i = 1, \ldots, n $. If $ 0<\alpha\leq1 $ and $(\lambda_1,\ldots,\lambda_n) \stackrel{\rm w}{\succeq}
(\lambda^{*}_{1}, \ldots , \lambda^{*}_{n})$, then $X_{n:n}\geq_{\rm st}X^{*}_{n:n}$.}
\end{theorem}
\begin{proof}
{\rm The distribution function of $ X_{n:n} $ can be written as
\begin{equation*}
G_{X_{n:n}}(x)= \prod_{i=1}^{n} [1-e^{1-(1+\lambda_i x)^{\alpha}}]^{\beta}.
\end{equation*}
Using Lemma \ref{mkl}, it is enough to show that the function $ G_{X_{n:n}}(x) $ is Schur-concave and increasing in $ \lambda_{i} $'s. The partial derivatives of $ G_{X_{n:n}}(x) $ with respect to $ \lambda_{i} $ is given by
\[
\frac{\partial G_{X_{n:n}}(x)}{\partial \lambda_{i}}=\frac{x\alpha \beta (1+\lambda_{i} x)^{\alpha-1}e^{1-(1+\lambda_{i} x)^{\alpha}}}{1-e^{1-(1+\lambda_{i} x)^{\alpha}}}G_{X_{n:n}}(x)\geq 0,
\]
So, $ G_{X_{n:n}}(x) $ is increasing in each $ \lambda_{i} $.

To prove its Schur-concavity, it follows from Theorem 3.A.4. in \cite{met} that we
have to show that for $ i\neq j $,
\[
(\lambda_{i}-\lambda_{j})\bigg(\frac{\partial G_{X_{n:n}}(x)}{\partial \lambda_{i}}-\frac{\partial G_{X_{n:n}}(x)}{\partial \lambda_{j}} \bigg)\leq 0,
\]
that is, for $ i\neq j $,
\begin{equation}\label{sq2}
(\lambda_{i}-\lambda_{j}) G_{X_{n:n}}(x)x\alpha \beta\bigg(\frac{(1+\lambda_{i} x)^{\alpha-1}e^{1-(1+\lambda_{i} x)^{\alpha}}}{1-e^{1-(1+\lambda_{i} x)^{\alpha}}}-\frac{(1+\lambda_{j} x)^{\alpha-1}e^{1-(1+\lambda_{j} x)^{\alpha}}}{1-e^{1-(1+\lambda_{j} x)^{\alpha}}}\bigg)\leq 0.
\end{equation}
It is easy to show that $ \frac{(1+\lambda_{i} x)^{\alpha-1}e^{1-(1+\lambda_{i} x)^{\alpha}}}{1-e^{1-(1+\lambda_{i} x)^{\alpha}}} $ is decreasing in $ \lambda_{i} $ for $ 0<\alpha\leq1 $, from which it follows
that \eqref{sq2} holds. This completes the proof of the result.}
\end{proof}
 The following theorem, generalizes Theorem \ref{e2} to a wider range of the scale parameters.
 \begin{theorem}\label{e22}
{\rm Let $ X_{1}, \ldots , X_{n} $ ($ X^{*}_{1}, \ldots , X^{*}_{n} $) be independent random variables with $ X_{i} \thicksim {\rm ENH}(\alpha, \lambda_{i}, \beta )$ ($ X^{*}_{i} \thicksim {\rm ENH}(\alpha, \lambda^{*}_{i}, \beta )$), $i = 1, \ldots, n $. If $ 0<\alpha\leq1 $ and $(\lambda_1,\ldots,\lambda_n) \leq
(\lambda^{*}_{1}, \ldots , \lambda^{*}_{n})$, that is, $ \lambda_{i}\leq \lambda^{*}_{i} $, $ i = 1, \ldots, n  $, then $X_{n:n}\geq_{\rm st}X^{*}_{n:n}$.}
\end{theorem}
\begin{proof}
By using the definition of the usual stochastic order and the
fact that $ G_{X_{n:n}}(x) $ is increasing in each $ \lambda_{i} $, the required
results follow easily.
\end{proof}
Now, we discuss stochastic comparison between the largest order statistics in the sense of the likelihood ratio order.
\begin{theorem}\label{e3}
{\rm Let $ X_{1}, \ldots , X_{n} $ be independent random variables with $ X_{i} \thicksim {\rm ENH}(\alpha, \lambda, \beta_{i} )$ and  $ X^{*}_{1}, \ldots , X^{*}_{n} $ be another set of independent random variables with $ X^{*}_{i} \thicksim {\rm ENH}(\alpha, \lambda, \beta^{*}_{i} ), 
i = 1, \ldots, n $. Then,
 $ X_{n:n}\geq_{\rm lr}X^{*}_{n:n} $  if and only if $ \sum_{i=1}^{n}\beta_{i}\geq \sum_{i=1}^{n}\beta^{*}_{i} $.}
\end{theorem}
\begin{proof}
$ X_{n:n} $ has the distribution function $  F_{n}(x) = (1-e^{1-(1+\lambda x)^{\alpha}})^{\sum_{i=1}^{n}\beta_{i}} $ and $ X^{*}_{n:n} $ has the distribution function $ G_{n}(x) =  (1-e^{1-(1+\lambda x)^{\alpha}})^{\sum_{i=1}^{n}\beta^{*}_{i}} $, and 
the corresponding density functions are
\[
f_{n}(x)=\frac{\alpha \sum_{i=1}^{n}\beta_{i}\lambda (1+\lambda x)^{\alpha-1}e^{1-(1+\lambda x)^{\alpha}}}{ (1-e^{1-(1+\lambda x)^{\alpha}})^{1-\sum_{i=1}^{n}\beta_{i}}},
\]
\[
g_{n}(x)=\frac{\alpha \sum_{i=1}^{n}\beta^{*}_{i}\lambda (1+\lambda x)^{\alpha-1}e^{1-(1+\lambda x)^{\alpha}}}{ (1-e^{1-(1+\lambda x)^{\alpha}})^{1-\sum_{i=1}^{n}\beta^{*}_{i}}}\quad x>0.
\]

Thus, for $ x>0 $, the ratio of the density functions of $ X_{n:n} $ and $ X^{*}_{n:n} $ is
$ \frac{f_{n}(x)}{g_{n}(x)}=\frac{\sum_{i=1}^{n}\beta_{i}}{\sum_{i=1}^{n}\beta^{*}_{i}}(1-e^{1-(1+\lambda x)^{\alpha}})^{\sum_{i=1}^{n}\beta_{i}-\sum_{i=1}^{n}\beta^{*}_{i}}. $
Therefore, $ \frac{f_{n}(x)}{g_{n}(x)} $ is increasing in $ x $ if and only if $ \sum_{i=1}^{n}\beta_{i}\geq \sum_{i=1}^{n}\beta^{*}_{i} $.
\end{proof}
\begin{theorem}\label{e4}
{\rm
Let $ X_{1}, \ldots , X_{n} $ be independent random variables with $  X_{i} \thicksim {\rm ENH}(\alpha, \lambda, \beta_{i} ),$
and  $ X^{*}_{1}, \ldots , X^{*}_{n}$ be another set of independent random variables with $ X^{*}_{i} \thicksim {\rm ENH}(\alpha, \lambda, \beta^{*}_{i} )$, 
$i = 1, \ldots, n $.  Then for $ \alpha<1 $ and $ \lambda>0 $, we have
\[
\sum_{i=1}^{n}\beta^{*}_{i}\leq \sum_{i=1}^{n}\beta_{i}<1 \Longrightarrow X_{n:n}\geq_{\rm disp}X^{*}_{n:n}.
\]
 }
\end{theorem}
\begin{proof}
From Nadarajah and Haghighi \cite{nh}, it is easy to show that $ X_{n:n} $ has decreasing hazard rate (DHR) property
if $ \sum_{i=1}^{n}\beta_{i}<1 $ and $ \alpha<1 $. Now, the desired result follows from Theorem 3.B.20 of \cite{ss} and the fact that
likelihood ratio order implies hazard rate order.
\end{proof}
\begin{theorem}\label{e5}
{\rm Suppose $ X_{1}, \ldots , X_{n} $ and $ X^{*}_{1}, \ldots , X^{*}_{n}$ are  independent samples with $  X_{i} \thicksim {\rm ENH}(\alpha_{1}, \lambda_{1}, \beta_{i} ),$ and $ X^{*}_{i} \thicksim {\rm ENH}(\alpha_{2}, \lambda_{2}, \beta^{*}_{i} ), i = 1, \ldots, n $. Then for $ \alpha_{1}\leq\alpha_{2} $ and $ \lambda>0 $, we have
\[
\sum_{i=1}^{n}\beta_{i} = \sum_{i=1}^{n}\beta^{*}_{i}  \Longrightarrow X_{n:n}\geq_{\rm c}X^{*}_{n:n}.
\]

}
\end{theorem}
\begin{proof}
$ X_{n:n} $ has the distribution function $  F_{X_{n:n}}(x) = (1-e^{1-(1+\lambda_1 x)^{\alpha_1}})^{\sum_{i=1}^{n}\beta_{i}} $ and $ X^{*}_{n:n} $ has the distribution function $ G_{X^{*}_{n:n}}(x) =  (1-e^{1-(1+\lambda_2 x)^{\alpha_2}})^{\sum_{i=1}^{n}\beta^{*}_{i}} $. Note that
\[
F^{-1}_{X_{n:n}}(x) = \dfrac{(1-\log(1-x^{\frac{1}{\sum_{i=1}^{n}\beta_{i}}}))^{\frac{1}{\alpha_1}}-1}{\lambda_1}
\]
and if
$ \sum_{i=1}^{n}\beta_{i} = \sum_{i=1}^{n}\beta^{*}_{i} $ then
\[
F^{-1}_{X_{n:n}}(G_{X^{*}_{n:n}}(x)) = \dfrac{(1+\lambda_2 x)^{\frac{\alpha_2}{\alpha_1}}-1}{\lambda_1} = h(x)
\]

In order to obtain the required result it suffices
to show that $ h(x) $ is convex in $ x $. The first and second partial derivatives $ h(x) $ with
respect to $ x $, respectively, are
\[
\frac{\partial h(x)}{\partial x} =\dfrac{\lambda_2}{\lambda_1}\dfrac{\alpha_2}{\alpha_1}(1+\lambda_2 x)^{\frac{\alpha_2}{\alpha_1}-1}
\]
and
\[
\frac{\partial^{2} h(x)}{\partial x^{2}} =\dfrac{\lambda_2}{\lambda_1}\dfrac{\alpha_2}{\alpha_1}\lambda_2 (\frac{\alpha_2}{\alpha_1}-1)(1+\lambda_2 x)^{\frac{\alpha_2}{\alpha_1}-2}.
\]
Thus, for any $ \alpha_{1}\leq\alpha_{2} $ we immediately observe that $ \frac{\partial^{2} h(x)}{\partial x^{2}} $ is positive, which
completes the proof of the theorem.
\end{proof}
As a direct consequence of Theorem \ref{e5}, we present the following corollary, which is of independent interest in economics.
\begin{corollary}
{\rm Suppose $ X_{1}, \ldots , X_{n} $ and $ X^{*}_{1}, \ldots , X^{*}_{n}$ are  independent samples with $  X_{i} \thicksim {\rm ENH}(\alpha_{1}, \lambda_{1}, \beta_{i} ),$ and $ X^{*}_{i} \thicksim {\rm ENH}(\alpha_{2}, \lambda_{2}, \beta^{*}_{i} ), i = 1, \ldots, n $. Then for $ \alpha_{1}\leq\alpha_{2} $ and $ \lambda>0 $, we have
\[
\sum_{i=1}^{n}\beta_{i} = \sum_{i=1}^{n}\beta^{*}_{i} \Longrightarrow X_{n:n}\geq_{\rm Lorenz}X^{*}_{n:n}.
\]

}
\end{corollary}

\subsection{Dependent samples with Archimedean structure}
Recently, some efforts are made to investigate stochastic comparisons on order statistics of random variables
with Archimedean copulas. See, for example, \cite{btr}, \cite{ll}, \cite{lfr} and
\cite{flld}. In this section, we derive new result on the usual stochastic order between extreme order statistics of two heterogeneous random vectors with the dependent components having exponentiated scale (ES) marginals and Archimedean copula structure. Recall that random variable $X$ belongs to the ES family of distributions if $ X\sim H(x) = [G(\lambda x)]^\alpha $, where $\alpha, \lambda>0 $ and $ G $ is called the baseline distribution and is an
absolutely continuous distribution function. We denote this family by $ {\rm ES}(\alpha, \lambda) $. 
Specifically, by $ \boldsymbol{X} \sim {\rm ES}(\boldsymbol{\alpha}, \boldsymbol{\lambda}, \phi ) $ we denote the sample having the Archimedean copula
with generator $ \phi $ and for $i=1,...,n$, $ X_i\sim ES(\alpha_i, \lambda_i ) $. In the following theorem, for the ES samples with Archimedean survival copulas, we present the usual
stochastic order on the smallest order statistics under weakly super-majorization order between shape parameters. For $ \boldsymbol{X} \sim {\rm ES}(\boldsymbol{\alpha}, \lambda, \phi ) $ and $ \boldsymbol{X^{*}} \sim {\rm ES}(\boldsymbol{\alpha^{*}}, \lambda, \phi) $, Bashkar et al. \cite{btr} in Theorem 4.1 showed that $ X_{1:n}\leq_{\rm st} X^{*}_{1:n} $ if $ \boldsymbol{\alpha}\mathop \succeq \limits^{{\mathop{\rm w}} } \boldsymbol{\alpha^{*}} $. Theorem \ref{ede1} generalizes the result of \cite{btr} to ES samples with not necessarily a common dependence structure.  The smallest order statistic $ X_{1:n} $ of the sample $ \boldsymbol{X} \sim {\rm ES}(\boldsymbol{\alpha}, \lambda, \phi ) $ gets survival function
\begin{equation}\label{for}
\bar{G}_{X_{1:n} }(x)  = \phi \big( \sum_{i=1}^{n}\psi(1-G^{\alpha_{i}}(\lambda x))\big) = J(\boldsymbol{\alpha}, \lambda, x, \phi)
\end{equation}

 \begin{theorem}\label{ede1}
{\rm For $ \boldsymbol{X} \sim {\rm ES}(\boldsymbol{\alpha}, \lambda, \phi_{1} ) $ and $ \boldsymbol{X^{*}} \sim {\rm ES}(\boldsymbol{\alpha^{*}}, \lambda, \phi_{2} ) $, if  $ \psi_2\circ\phi_1 $ is super-additive, then $ \boldsymbol{\alpha} \stackrel{\rm w}{\succeq}
\boldsymbol{\alpha^{*}} $ implies $X_{1:n}\leq_{\rm st}X^{*}_{1:n}$.
}
\end{theorem}
\begin{proof}
According to Equation \eqref{for}, $ X_{1:n} $ and $ X^{*}_{1:n} $ have their respective survival functions $ J(\boldsymbol{\alpha}, \lambda, x, \phi_1) $ and $ J(\boldsymbol{\alpha^{*}}, \lambda, x, \phi_2) $ for $ x \geq 0 $.

First we show that $ J(\boldsymbol{\alpha}, \lambda, x, \phi_1) $ is increasing and Schur-concave function of $ \alpha_i, i=1,\ldots,n$. Since $ \phi_1 $ is decreasing, we have
\[
\dfrac{\partial J(\boldsymbol{\alpha}, \lambda, x, \phi_1)}{\partial \alpha_i} =- \dfrac{F^{\alpha_{i}}(\lambda x)\log(F(\lambda x))\phi_1^{\prime}\big(\sum_{i=1}^{n}\psi(1- F^{\alpha_{i}}(\lambda x))\big)}{\phi_1^{\prime}(\psi(1- F^{\alpha_{i}}(\lambda x)))}\geq 0,
\]
\[
 \text{for all}\quad x>0,
\]	
That is, $ J(\boldsymbol{\alpha}, \lambda, x, \phi_1) $ is increasing in $ \alpha_{i} $ for $ i = 1, \ldots , n $.

To prove its Schur-concavety, it follows from Theorem 3.A.4. in \cite{met} that we
have to show that for $ i\neq j $,
\[
(\alpha_i - \alpha_j) \big(\dfrac{\partial J(\boldsymbol{\alpha}, \lambda, x, \phi_1)}{\partial \alpha_i}-\dfrac{\partial J(\boldsymbol{\alpha}, \lambda, x, \phi_1)}{\partial \alpha_j}\big) \leq 0,
\]
that is, for $ i\neq j $,
\begin{align*}
-\log(F(\lambda x))\phi_1^{\prime}\big(\sum_{i=1}^{n}\psi_1(1- F^{\alpha_{i}}(\lambda x)))\big)(\alpha_{i}-\alpha_{j})
 \end{align*}
 \begin{align}\label{de2}
\bigg(\dfrac{F^{\alpha_{i}}(\lambda x)}{\phi_1^{\prime}(\psi_1(1- F^{\alpha_{i}}(\lambda x)))}-\dfrac{F^{\alpha_{j}}(\lambda x)}{\phi_1^{\prime}(\psi_1(1- F^{\alpha_{j}}(\lambda x)))}\bigg)\leq 0.
\end{align}
Now, let us consider the function $ g(\alpha) = \dfrac{F^{\alpha}(\lambda x)}{\phi^{\prime}(\psi(1- F^{\alpha}(\lambda x)))} $. Taking derivative with respect to $ \alpha $, we get 
\[
g^{\prime}(\alpha) \mathop  = \limits^{{\mathop{\rm sgn}} } F^{\alpha}(\lambda x)\log(F(\lambda x))\phi^{\prime}(\psi(1- F^{\alpha}(\lambda x)))
\]
\[
+\dfrac{F^{2\alpha}(\lambda x)\log(F(\lambda x))}{\phi^{\prime}(\psi(1- F^{\alpha}(\lambda x)))}\phi^{\prime \prime}(\psi(1- F^{\alpha}(\lambda x)))\geq0.
\]
Thus, $ g(\alpha) $ is increasing with respect
to $ \alpha $, from which it follows
that \eqref{de2} holds. 
According to Lemma \ref{mkl} $ \boldsymbol{\alpha} \stackrel{\rm w}{\succeq}
\boldsymbol{\alpha^{*}} $
implies $ J(\boldsymbol{\alpha}, \lambda, x, \phi_1)\leq J(\boldsymbol{\alpha^{*}}, \lambda, x, \phi_1) $.
On the other hand, since $ \psi_2\circ\phi_1 $ is super-additive
by Lemma A.1. of \cite{lfr}, we have $ J(\boldsymbol{\alpha^{*}}, \lambda, x, \phi_1) \leq J(\boldsymbol{\alpha^{*}}, \lambda, x, \phi_2) $. So, it holds that
\[
J(\boldsymbol{\alpha}, \lambda, x, \phi_1)\leq J(\boldsymbol{\alpha^{*}}, \lambda, x, \phi_1) \leq J(\boldsymbol{\alpha^{*}}, \lambda, x, \phi_2).
\]
That is, $ X_{1:n}\leq_{\rm st} X^{*}_{1:n} $.
\end{proof}
Note that if in Theorem \ref{ede1}, we take $ \lambda =1 $, then we get the following result for the proportional reversed hazards (PRH) model.

\begin{corollary}
{\rm Suppose $ \boldsymbol{X} \sim PRH(\boldsymbol{\alpha}, \phi_{1} ) $ and $ \boldsymbol{X^{*}} \sim PRH(\boldsymbol{\alpha^{*}}, \phi_{2} )$  and $ \phi_2\circ\psi_1 $ is super-additive. Then $ \boldsymbol{\alpha} \stackrel{\rm w}{\succeq}
\boldsymbol{\alpha^{*}} $ implies $X_{1:n}\leq_{\rm st}X^{*}_{1:n}$.}
\end{corollary}
The following corollary immediately follows from the above theorem.
\begin{corollary}
{\rm Suppose $ \boldsymbol{X} \sim ENH(\alpha,\lambda,\boldsymbol{\beta}, \phi_{1} ) $ and $ \boldsymbol{X^{*}} \sim ENH(\alpha,\lambda,\boldsymbol{\beta^{*}}, \phi_{2} )$  and $ \phi_2\circ\psi_1 $ is super-additive. Then $ \boldsymbol{\beta} \stackrel{\rm w}{\succeq}
\boldsymbol{\beta^{*}} $ implies $X_{1:n}\leq_{\rm st}X^{*}_{1:n}$.}
\end{corollary}
\section{Conclusions}
The failure rate function of the ENH model can be
constant, decreasing, increasing, upside-down bathtub (unimodal) and bathtub-shaped. Due to the great flexibility of the
failure rate function of this model, it thus provides a good alternative to many existing life distributions in modeling
positive real data sets in practice. In particular, it can be an interesting alternative to the well-known three-parameter
EW distribution. In this paper, we discussed stochastic comparisons of parallel  systems with independent  heterogeneous ENH
components in terms of the usual stochastic order,
dispersive order, convex transform order and the likelihood ratio order. In the presence of the Archimedean copula, we studied stochastic comparison of series dependent systems in terms of the usual stochastic order.






\end{document}